\documentclass[smallextended,final]{svjour3}
\usepackage[numbers]{natbib}
\usepackage{epsf,latexsym,amsmath,amsfonts,amssymb,subfigure,mathrsfs,graphicx,stmaryrd,eufrak,esint,amscd}
\usepackage{color}

\newcommand{\abs}[1]{\lvert#1\rvert}

\newcommand{\Lr}[1]{\left(#1\right)}
\newcommand{\lr}[1]{\bigl(#1\bigr)}
\newcommand{\set}[2]{\{\,#1\,\mid\,#2\,\}}
\newcommand{\diff}[2]{\frac{\partial #1}{\partial #2}}
\newcommand{\jump}[1]{[\![#1]\!]}
\newcommand{\nm}[2]{\|#1\|_{#2}}
\newcommand{\wnm}[1]{|\!|\!|#1|\!|\!|_{\iota,h}}
\def\ldof{d^{(l)}}
\def\mdof{d^{(m)}}

\newcommand{\mc}[1]{\mathcal{#1}}

\def\mr{\mathrm}

\def\ka{\kappa}
\def\al{\alpha}
\def\eps{\epsilon}

\def\de{\delta}
\def\na{\nabla}
\def\pa{\partial}

\def\Om{\Omega}
\def\lam{\lambda}

\def\C{\mathbb C}

\def\dx{\,\mathrm{d}x}
\def\dtau{\,\mathrm{d}\tau}

\def\D{\mathbb D}

\newcommand{\divop}{\na\cdot}
\newcommand{\curl}{\text{curl}}
\newcommand{\trop}{\text{tr}}
\begin{document}
\title{Two Robust Nonconforming H$^2-$elements for Linear Strain Gradient Elasticity}

\author{Hongliang Li \and Pingbing Ming \and Zhong-ci Shi}
\institute{H.L. Li \at Institute of Electronic Engineering, China Academy of Engineering Physics, Mianyang, 621900, China\\
  \email{lihongliang@mtrc.ac.cn}\\
  \and P.B. Ming \and Z.-C. Shi \at The State Key Laboratory of Scientific and Engineering Computing,\\
Academy of Mathematics and Systems Science, Chinese Academy of Sciences, \\
  No. 55, Zhong-Guan-Cun East Road, Beijing 100190, China\\
  University of Chinese Academy of Sciences, Beijing, 100049, China\\
  \email{mpb@lsec.cc.ac.cn}\and\email{shi@lsec.cc.ac.cn}}

\date{Received: January 20: 2016 / Accepted:}
\maketitle

\begin{abstract}
We propose two finite elements to approximate a boundary value problem arising from
strain gradient elasticity, which is a high order perturbation of the linearized
elastic system. Our elements are H$^2-$nonconforming while H$^1-$conforming. We show both
elements converge in the energy norm uniformly with respect to the perturbation parameter.
\keywords{Strain gradient elasticity \and nonconforming finite elements \and uniform error estimates}
\subclass{MSC Primary 65N30, 65N15; Secondary 74K20}
\end{abstract}
\section{Introduction}
Strain gradient theory, which introduces the high order strain and microscopic parameter into the strain energy density, is one of the most successful approach to characterize the strong size effect of the heterogeneous materials~\cite{MaClarke:1995}. The origin of this theory can be traced back to Cosserat brothers' celebrated work~\cite {Cosserat:1909}. 
Further development is related to Mindlin's work on microstructure in linear elasticity~\cite{Mindlin:1964,MindlinEshel:1968}. However, Mindlin's theory is less attractive in practice because it contains too many parameters. Based on Mindlin's work, Aifantis at al~\cite{Altan:1992,Ru:1993} proposed a linear strain gradient elastic model with only one microscopic parameter. This simplified strain gradient theory successfully eliminated the strain singularity of the brittle crack tip field~\cite{Exadaktylos:1996}.

The strain gradient elastic model of Aifantis' is a perturbed elliptic system of fourth order from the view point of mathematics. To discrtize this model by a finite element method, a natural choice is C$^1$ finite elements such as Argyris triangle~\cite{ArgyrisFriedScharpf:1968} and Bell's triangle~\cite{Bell:1969} because this model contains the gradient of strain. We refer to~\cite{Zerovs:2001,zerovs:20091,Zerovs:20092,Akarapu:2006} for works in this direction. Alternative approach such as mixed finite element has been employed to solve this model~\cite{Amanatidou:2002}. A drawback of both the conforming finite element method and mixed finite element method is that the number of the degrees of freedom is extremely large and high order
polynomial has to be used in the basis function, which is more pronounced for three dimensional problems; See e.g.,
the finite element for three-dimensional strain gradient model proposed in~\cite{Zerovs:20092}
contains $192$ degrees of freedom for the local finite element space.

A common approach to avoid such difficult is to use the nonconforming finite element. For scalar version of such problem, there are a lot of work since the original contribution~\cite{Tai:2001}, and we refer to~\cite{Brenner:2011,GuzmanLeykekhmanNeilan:2012,Xie:2013} and the references therein for recent progress. The situation is different for the strain gradient elastic model. The well-posedness of the corresponding boundary value problem hinges on a Korn-like inequality, which will be dubbed as H$^2$-Korn's inequality. Therefore, a discrete H$^2$-Korn's inequality has to be satisfied for any reasonable nonconforming finite element approximation. We prove a H$^2$-Korn's inequality for piecewise vector fields as {\sc Brenner's} seminal H$^1$ Korn's inequality~\cite{Brenner:2004}. Motivated by this discrete Korn's inequality, we propose two nonconforming H$^2-$finite elements, which are H$^1$-conforming. Both elements satisfy the discrete H$^2$-Korn's inequality. We prove that both elements converge in energy norm uniformly with respect to the small perturbation parameter. Numerical results also confirm the theoretic results.

It is worth mentioning that {\sc Soh and Chen}~\cite{Chen:2004} constructed several noncomforming finite elements for this strain gradient elastic model. Some of them exhibited excellent numerical performance. Their motivation is the so-called $C^{\,0-1}$ patch test, which is obviously different from ours. It is unclear whether their
elements are robust with respect to the small perturbation parameter, which may be an interesting topic for
further study. As to various numerical methods based on reformulations of the strain gradient elastic model, we refer to~\cite{Askes:2002,Wei:2006} and the references therein.

The remaining part of the paper is organized as follows. In the next part, we introduce the linear strain gradient elasticity and prove the well-posedness of its Dirichlet boundary value problem by establishing a H$^2$-Korn inequality. A discrete H$^2$-Korn inequality is proved in Section~\ref{sec:fem}, and two finite elements are constructed and analyzed in this part. The numerical results can be found in Section~\ref{sec:numer}.

Throughout this paper, the generic constant $C$ may differ at different occurrences, while it is independent of the microscopic parameter $\iota$ and the mesh size $h$.
\section{The Korn's inequality of strain gradient elasticity}\label{sec:linearsgt}
The space $L^2(\Om)$ of the square-integrable
functions defined on a bounded polygon $\Om$ is equipped with the inner product $(\cdot,\cdot)$ and the norm $\nm{\cdot}{L^2(\Om)}$. Let $H^m(\Om)$ be the
standard Sobolev space~\cite{AdamsFournier:2003} and
 \[
 \nm{v}{H^m(\Om)}=\sum_{k=0}^m\abs{v}_{H^k(\Om)}^2\quad
 \text{and}\quad\abs{v}_{H^k(\Om)}^2=\int_{\Om}\sum_{\abs{\al}=k}\abs{\na^\al v}^2\dx,
 \]
 where $\al=(\al_1,\al_2)$ is a multi-index whose components $\al_i$ are nonnegative integers, $\abs{\al}=\al_1+\al_2$ and $\na^\al=\pa^{\abs{\al}}/\pa x_1^{\al_1}\pa x_2^{\al_2}$. We may drop $\Om$ in the Sobolev norm $\nm{\cdot}{H^m(\Om)}$
 when there is no confusion may occur. The space $H_0^m(\Om)$ is the closure in $H^m(\Om)$ of $C_0^\infty(\Om)$. In particular,
 \begin{align*}
 H_0^1(\Om){:}&=\set{v\in H^1(\Om)}{v=0\text{\;on\;}\pa\Om},\\
 H_0^2(\Om){:}&=\set{v\in H^1(\Om)}{v=\pa_n v=0\text{\;on\;}\pa\Om},
 \end{align*}
 where $\pa_n v$ is the normal derivative of $v$. Equally, $\pa_t v$ denotes the tangential derivative of $v$.
 The summation convention is used for repeated indices. A comma followed by a subscript, say $i$, denotes partial differentiation with respect to the spatial variables $x_i$, i.e., $v_{,i}=\pa v/\pa x_i$.

For any vector-valued function $v$, its gradient is a matrix-valued
function with components $(\na v)_{ij}=\pa v_i/\pa x_j$. The symmetric part of a gradient field is also a matrix-valued function defined by
\[
\eps(v)=\dfrac12(\na v+[\na v]^T).
\]
The anti-symmetric part of a gradient field is defined as \(\na^{\text{a}}v=\na v-\eps(v)\).
The divergence operator applying to a vector field is defined as the trace of $\na v$, i.e., $\divop v=\trop{\na v}=\pa v_i/\pa x_i$.
The Sobolev spaces $[H^m(\Om)]^2, [H_0^m(\Om)]^2$ and  $[L^2(\Om)]^2$ of a vector field can be defined
in a similar manner as their scalar counterparts, this rule equally applies to their inner products and their norms. For the $m$-th order tensors $A$ and $B$, we define the inner product as $A:B=\sum_{i_1,\cdots,i_m}A_{i_1}B_{i_1}
\cdots A_{i_m}B_{i_m}$.
\subsection{Strain gradient elastic model and H$^2-$Korn inequality}
The strain gradient elastic model in~\cite{Altan:1992,Ru:1993,Exadaktylos:1996} is described by the following boundary
value problem: For $u$ the displacement vector that solves
\begin{equation}\label{eq:sgbvp}
\left\{
\begin{aligned}
(\iota^2\triangle-I)\Lr{\mu\triangle u+(\lambda+\mu)\na\divop u}&=f,\quad&&\text{in\;}\Om,\\
u=\pa_n u&=0,\quad&&\text{on\;}\pa\Om.
\end{aligned}\right.
\end{equation}
Here $\lam$ and $\mu$ are the Lam\'e constants, and $\iota$ is the microscopic
parameter such that $0<\iota\le 1$. In particular, we are interested in the regime when $\iota$ is close
to zero. The above boundary value problem may be rewritten into the following variational problem: Find $u\in[H^2_0(\Om)]^2$ such that
\begin{equation}\label{variation:eq}
a(u,v)=(f,v)\quad\text{for all\quad} v\in [H_0^2(\Om)]^2,
\end{equation}
where
\[
a(u,v)=(\C\eps(u),\eps(v))+(\D\na\eps(u),\na\eps(v)),
\]
and the fourth-order tensors $\C$ and the sixth-order tensor $\D$ are defined by
\[
\C_{ijkl}=\lam\de_{ij}\de_{kl}+2\mu\de_{ik}\de_{jl}\quad\text{and}\quad
\D_{ijklmn}=\iota^2\Lr{\lam\de_{il}\de_{jk}\de_{mn}+2\mu\de_{il}\de_{jm}\de_{ln}},
\]
respectively. Here $\delta_{ij}$ is the Kronecker delta function. The third-order tensor $\na\eps(v)$ is defined as $(\na\eps(v))_{ijk}=\eps_{jk,i}$. We only consider the clamped boundary condition in this paper, the discussion on
other boundary conditions can be found in~\cite{Altan:1992,Ru:1993,Exadaktylos:1996}.

The variational problem~\eqref{variation:eq} is well-posed if and only if the bilinear form $a(\cdot,\cdot)$ is coercive over $[H_0^2(\Om)]^2$.
\begin{theorem}\label{thm:korn}
For any $v\in [H_0^2(\Om)]^2$, there holds
\begin{equation}\label{Korn}
C(\Om)\Lr{\nm{v}{H^1}^2+\iota^2\abs{v}_{H^2}^2}\le a(v,v)\leq 2(\lam+\mu)\Lr{\nm{v}{H^1}^2+\iota^2\abs{v}_{H^2}^2},
\end{equation}
where $C(\Om)$ denpends only on $\mu$ and the constant $C_p$ in the following {\em Poincar\'e} inequality,
\[
\nm{v}{L^2}\le C_p\nm{\na v}{L^2}.
\]
\end{theorem}

The proof of this theorem essentially depends on the first {\rm Korn's} inequality~\cite{Korn:1908,Korn:1909}. For any $v\in [H_0^1(\Om)]^2$, there holds
\begin{equation}\label{eq:1stkorn}
2\nm{\eps(v)}{L^2}^2\ge\nm{\na v}{L^2}^2.
\end{equation}
The proof of this inequality follows from the following identity
\[
\abs{\eps(v)}^2-\abs{\na^{\text{a}}v}^2=\abs{\divop v}^2+\divop[(v\cdot\na)v-v(\divop v)]
\]
with the usual notation
\[
v\cdot\na=\sum_{i=1}^2v_i\diff{}{x_i}.
\]
Indeed, by the fact $v=0$ on $\pa\Om$, the above identity and divergence theorem imply
\[
\int_{\Om}\abs{\na^{\text{a}}v}^2\dx=\int_{\Om}\abs{\eps(v)}^2\dx-\int_{\Om}\abs{\divop v}^2\dx\le\int_{\Om}\abs{\eps(v)}^2\dx,
\]
which implies the first Korn's inequality~\eqref{eq:1stkorn} by using
the algebraic identity
\[
\abs{\na v}^2=\abs{\eps(v)}^2+\abs{\na^{\text{a}}v}^2.
\]
\vskip .5cm
\noindent
{\em Proof of Theorem~\ref{thm:korn}\;}
By definition, we write
\[
a(v,v)=2\mu\nm{\eps(v)}{L^2}^2+\lam\nm{\divop v}{L^2}^2+\iota^2\Lr{2\mu\nm{\na\eps(v)}{L^2}^2+\lam\nm{\na\divop v}{L^2}^2}.
\]
The upper bound in~\eqref{Korn} immediately follows by noting
\[
\nm{\eps(v)}{L^2}^2\le\nm{\na v}{L^2}^2,\quad\text{and}\quad
\nm{\divop v}{L^2}^2\le 2\nm{\na v}{L^2}^2.
\]

For any $v\in[H_0^2(\Om)]^2$, we have $\pa_i v\in[H_0^1(\Om)]^2$ for $i=1,2$, we apply the first Korn's inequality~\eqref{eq:1stkorn}
to the vector field $\pa_i v$ and obtain
\[
2\nm{\eps(\pa_i v)}{L^2}^2\ge\nm{\na\pa_i v}{L^2}^2.
\]
Using the fact that the strain operator $\eps$ and the partial gradient operator $\pa_i$ commute, we rewrite the above inequality as
\[
2\nm{\na\eps(v)}{L^2}^2=2\sum_{i=1}^2\nm{\pa_i\eps(v)}{L^2}^2
=2\sum_{i=1}^2\nm{\eps(\pa_i v)}{L^2}^2\ge\sum_{i=1}^2\nm{\na\pa_i v}{L^2}^2=\nm{\na^2 v}{L^2}^2.
\]
Therefore,
\[
a(v,v)\ge\mu\Lr{\abs{v}_{H^1}^2+\iota^2\abs{v}_{H^2}^2},
\]
which together with the {\em Poincar\'e} inequality leads to
the lower bound in~\eqref{Korn}.
\qed

Proceeding along the same line in~\cite[\S 5]{Tai:2001}, we may prove the following regularity results
for the solution of Problem~\eqref{variation:eq}.
\begin{lemma}\label{lemma:reg}
There exists $C$ that may depend on $\Om$ but independent of $\iota$ such that
\begin{equation}\label{eq:regularity}
\abs{u}_{H^2}+\iota\abs{u}_{H^3}\le C\iota^{-1/2}\nm{f}{L^2},
\end{equation}
and
\begin{equation}\label{eq:limiterr}
\nm{u-u^0}{H^1}\le C\iota^{1/2}\nm{f}{L^2},
\end{equation}
where $u^0\in [H_0^1(\Om)]^2$ satisfies
\begin{equation}\label{eq:limit}
(\C\eps(u^0),\eps(v))=(f,v)\qquad \text{for all\quad}v\in[H_0^1(\Om)]^2.
\end{equation}
\end{lemma}
\section{The nonconforming finite elements}~\label{sec:fem}
In this part, we introduce two nonconforming finite elements
to approximate the variational problem~\eqref{variation:eq}.
Let $\mc{T}_h$ be a triangulation of $\Omega$ with maximum mesh size $h$. We assume all elements in $\mc{T}_h$ is shape-regular in the sense of Ciarlet and Raviart~\cite{Ciarlet:1978}. Denote the set of all the edges in $\mc{T}_h$ as $\mc{S}(\Omega,\mc{T}_h)$. The space of piecewise $[H^m(\Om,\mc{T}_h)]^2$ vector fields is defined by
\[
[H^m(\Omega,\mc{T}_h)]^2{:}=\set{v\in [L^2(\Omega)]^2}{v|_T\in[H^m(T)]^2,\quad\forall T\in\mc{T}_h},
\]
which is equipped with the broken norm
\[
\nm{v}{H_h^k}{:}=\nm{v}{L^2}+\sum_{k=1}^m\nm{\na^k_h v}{L^2},
\]
where
\[
\nm{\na^k_h v}{L^2}^2=\sum_{T\in\mc{T}_h}\nm{\na^k v}{L^2(T)}^2
\]
with $(\na^k_h v)|_T=(\na^k v)|_T$. Moreover, $\eps_h(v)=(\na_h v+[\na_h v]^T)/2$.

{\sc Brenner}~\cite{Brenner:2004} established a discrete Korn inequality for any piecewise $H^1$ vector fields with weak linear continuity across the common surface between
two adjacent elements, i.e., for any $v\in [H^1(\Om,\mc{T}_h)]^d$ with $d=2,3$ satisfying
\[
\int_{e}\jump{v}\cdot p\dtau=0,\quad p\in[P_1(e)]^d,e\in\mc{S}(\Om,\mc{T}_h).
\]
There exists a constant $C$ depends on $\Om$ and $\mc{T}_h$ but independent of $h$ such that
\begin{equation}\label{eq:h1korn}
\nm{v}{H_h^1}\le C\Lr{\nm{v}{L^2}+\nm{\eps_h(v)}{L^2}}.
\end{equation}
Here $[P_1(e)]^d$ is the linear vector field over $e$ and $\jump{v}$ denotes the jump of $v$ across $e$ with $e$ an edge for $d=2$ and a face for $d=3$. This inequality is fundamental to the well-posedness of the discrete problems arising from nonconforming finite element and discontinuous Galerkin method approximation of the linearized elasticity model and Reissner-Mindlin plate model; See~\cite{Falk:1991},~\cite{MingShi:2001} and~\cite{HansboLarson:2002}.

{\sc Mardal and Winther}~\cite{Winther:2006} improved the above inequality by replacing $[P_1(e)]^d$ by its subspace $[P_{1,-}(e)]^d$ given by
\[
[P_{1,-}(e)]^d{:}=\set{v\in [P_1(e)]^d}{v\cdot t\in\text{RM}(e)},
\]
where $t$ is the tangential vector of edge $e$, and $\text{RM}(e)$ is the infinitesimal rigid motion on $e$. In fact, they have proved
\begin{equation}\label{eq:discretekorn}
\abs{v}_{H_h^1}^2\le C\Lr{\|\eps_h(v)\|_{L^2}^2+\nm{v}{L^2}^2
+\sum_{e\in\mc{S}(\Omega,\mc{T}_h)}h_{e}^{-1}\nm{\jump{\Pi_{e}v}}{L^2(e)}^2},
\end{equation}
where $\Pi_{e}:[L^2(e)]^d\mapsto [P_{1,-}(e)]^d$ is the $L^2$ projection.

Our result is an H$^2$ analog of the discrete Korn's inequality~\eqref{eq:discretekorn}.
\begin{theorem}\label{thm:hkorn}
For any $v\in[H^2(\Omega,\mc{T}_h)]^2$, there exits $C$ that depends on $\Om$ and $\mc{T}_h$ but independent of $h$ such that
\begin{equation}\label{discrete_Korn}
\begin{aligned}
\nm{v}{H_h^2}^2&\le C\bigg(\nm{\na_h\eps_h(v)}{L^2}^2+\nm{\eps_h(v)}{L^2}^2+\nm{v}{L^2}^2
+\sum_{e\in\mc{S}(\Omega,\mc{T}_h)}h_{e}^{-1}\nm{\jump{\Pi_{e}v}}{L^2(e)}^2\\
&\qquad\quad+\sum_{i=1}^2\sum_{e\in\mc{S}(\Omega,\mc{T}_h)}h_{e}^{-1}\nm{\jump{\Pi_{e}(\pa_i v)}}{L^2(e)}^2\bigg).
\end{aligned}
\end{equation}
\end{theorem}

The proof follows essentially the same line that leads to Theorem~\ref{thm:korn}.
\begin{proof}
For any $v\in[H^2(\Omega,\mc{T}_h)]^2$, it is clear that $\pa_i v\in[H^1(\Om,\mc{T}_h)]^2$ for $i=1,2$.
Applying the discrete Korn's inequality~\eqref{eq:discretekorn} to each $\pa_i v$, we obtain
\begin{align*}
\abs{v}_{H_h^2}^2&=\abs{\pa_1v}_{H_h^1}^2+\abs{\pa_2v}_{H_h^1}^2\\
&\le C\sum_{i=1}^2\Lr{\nm{\eps_h(\pa_iv)}{L^2}^2+\nm{\pa_iv}{L^2}^2
+\sum_{e\in\mc{S}(\Omega,\mc{T}_h)}h_{e}^{-1}\nm{\jump{\Pi_{e}(\pa_iv)}}{L^2(e)}^2}\\
&=C\Lr{\nm{\na_h\eps_h(v)}{L^2}^2+\nm{\na_h v}{L^2}^2
+\sum_{i=1}^2\sum_{e\in\mc{S}(\Om,\mc{T}_h)}h_{e}^{-1}\nm{\jump{\Pi_{e}(\pa_i v)}}{L^2(e)}^2}.
\end{align*}
Invoking~\eqref{eq:discretekorn} once again, we get~\eqref{discrete_Korn}.\qed
\end{proof}

Motivated by the discrete Korn's inequality~\eqref{discrete_Korn}, we construct two new finite elements
that are H$^1-$conforming but H$^2-$nonconforming elements. For such elements, the continuity of the tangential derivatives are automatically satisfied, and we only need to deal with the weak continuity of the normal derivative. The finite element space is defined as
\[
V_h{:}=\set{v\in [H_0^1(\Om)]^2}{v|_T\in W(T)\;\text{for all}\;T\in\mc{T}_h}.
\]
We shall specify two local finite element spaces $W(T)$ in the next two parts.

Given $V_h$, we find $u_h\in V_h$ such that
\begin{equation}\label{eq:discvara}
a_h(u_h,v)=(f,v)\quad\text{for all\quad} v\in V_h,
\end{equation}
where the bilinear form $a_h$ is defined for any $v,w\in V_h$ as
\[
a_h(v,w){:}=(\C\eps(v),\eps(w))+(\D\na_h\eps(v),\na_h\eps(w)),
\]
where the second term is defined in a piecewise manner as
\[
(\D\na_h\eps(v),\na_h\eps(w)){:}=\sum_{T\in\mc{T}_h}\int_T\D\na\eps(v)\na\eps(w)\dx.
\]
\subsection{The first nonconforming element}
Define
\begin{equation}\label{element:1}
W(T){:}=[P_2(T)]^2\oplus bP_2^\ast(T),
\end{equation}
where $P_2(T)$ is the quadratic Lagrange element, and $b=\lam_1\lam_2\lam_3$ is the cubic bubble function, and $P_2^\ast(T)\subset[P_2(T)]^2$ is defined as
\[
P_2^\ast(T){:}=\set{v\in[P_2(T)]^2}{v\cdot n|_e\in P_1(e)\quad\text{for all\quad}e\in\pa T}.
\]

Next lemma gives the degrees of freedom of this element, which is graphically shown in Figure~\ref{element}, and we prove
that the degrees of freedom is $W(T)-$unisolvent.
\begin{figure}
\centering
\includegraphics[width=6cm, height=5cm]{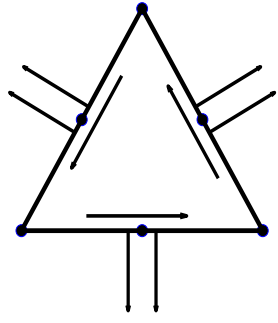}
\caption{The degrees of freedom are point evaluations at the vertex and midpoint of each edge, are the moments of the normal derivative of the normal component against $P_1$ over each edge, are the moments of the normal
derivative of the tangential component along each edge.}\label{element}
\end{figure}
\begin{lemma}\label{lem:unisolvent}
The dimension of $W(T)$ is $21$. Any $w\in W(T)$ is uniquely determined by the following degrees of freedom:
\begin{enumerate}
\item The values of $w$ at the corners and edge midpoints;
\item The moments $\int_e\pa_n (w\cdot t)\dtau$ and $\int_e\pa_n (w\cdot n) \tau^k\dtau$ for $k=0,1$
and for all $e\in\pa T$.
\end{enumerate}
\end{lemma}

\begin{proof}
Since $[P_2(T)]^2\cap b P^{\ast}_2(T)=\{0\}$ and $\mr{dim}\;P^{\ast}_2(T)\geq9$, we conclude $\mr{dim}\;W(T)\ge 21$.
It suffices to show that a function $w\in W(T)$ vanishes if all the degrees of freedom
are zeros. Note that $w|_e\in [P_2(e)]^2$, with three roots
on edge $e$, then we must have $w|_{\pa T}=0$. Therefore, we may write $w=bp$ with $p\in P_2^\ast(T)$.
Let $e$ be a fixed edge of $T$, and denote $b=\lam_e\lam_+\lam_{-}$ with $\lam_e$ the barycentric coordinate functions
such that $\lam_e\equiv 0$ on $e$, while $\lam_+$ and $\lam_{-}$ the remaining two barycentric coordinate functions. Furthermore, $(\na w)|_e=(p\lam_+\lam_{-})|_e\na\lam_e$. Note that $\lam_+\lam_{-}\pa_n\lam_e$
is strictly negative in the interior of $e$. Therefore, the condition
\[
0=\int_e\pa_n(w\cdot n)\tau^k\dtau=\int_e\lam_+\lam_{-}\diff{\lam_e}{n}p\cdot n\tau^k\dtau
\]
implies that for any $e\in\pa T$ and $k=0,1$,
\begin{equation}\label{eq:1stnormalconstraint}
\int_ep\cdot n\tau^k\dtau=0.
\end{equation}
Proceeding along the same line, we obtain, for any $e\in\pa T$,
\begin{equation}\label{eq:1sttangential}
\int_ep\cdot t\dtau=0.
\end{equation}
Furthermore, using the fact that $p\cdot n\in P_1(e)$ and~\eqref{eq:1stnormalconstraint}, we conclude
that $p\cdot n\equiv 0$ over $\pa T$.

Assume that
\[
p=\Lr{\sum_{\abs{\al}=2}a_{\al}\lam_1^{\al_1}\lam_2^{\al_2}\lam_3^{\al_3},\sum_{\abs{\al}=2}b_{\al}\lam_1^{\al_1}\lam_2^{\al_2}\lam_3^{\al_3}}^T,
\]
where $\al=(\al_1,\al_2,\al_3)$ whose components $\al_i$ are nonnegative integers and $\abs{\al}=\al_1+\al_2+\al_3$.

Using $(p\cdot n_i)|_{e_i}\equiv 0$, we obtain, for $j=1,2,3$,
\begin{equation}\label{eq:orth1}
(a_{\al},b_{\al})\cdot n_j=0\quad\text{with\quad}\al_j=0.
\end{equation}
This means that, for each $\al$ with one component equals to $2$, the vector $(a_{\al},b_{\al})$ is orthogonal to the normal directions of two different edges, which immediately implies that such vector $(a_{\al},b_{\al})$ must be zero.

Next, using the fact that $\int_{e_i}(p\cdot t)\dtau\equiv 0$, we obtain, for $j=1,2,3$,
\[
(a_{\al},b_{\al})\cdot t_j=0,\quad\text{with\quad} \al_j=0,\al_k=1\quad\text{for\quad}k\not=j.
\]
Invoking~\eqref{eq:orth1} once again, we conclude that, for each $\al$ with only one zero component, $(a_{\al},b_{\al})$ is orthogonal to
both the normal direction and the tangential direction of the edge indexed with the zero component of $\al$, which must vanish identically. Therefore, all $a_{\al}$ and $b_{\al}$ are zeros, and hence $p\equiv 0$, equivalently, $w\equiv 0$. This completes the proof.\qed
\end{proof}

Using the degrees of freedom given in Lemma~\ref{lem:unisolvent}, we may define a local interpolation operator
$\Pi_T:H^2(T)\mapsto W(T)$. The next lemma shows that this operator locally preserves quadratics.
\begin{lemma}\label{lem:interpolate}
\begin{equation}\label{eq:p2invariance}
\Pi_T v=v,\quad v\in [P_2(T)]^2.
\end{equation}
\end{lemma}

\begin{proof}
Let $\lr{T,W(T),\Sigma(T)}$ be the finite element triple with $\Sigma(T)$ the degrees of freedom. By construction, $\Sigma(T)$ takes the form
as
\[
\Sigma(T)=\{\ldof_1,\cdots,\ldof_{12},\mdof_1,\cdots,
\mdof_9\},
\]
where $\{\ldof_i\}_{i=1}^{12}$ are the nodal type degrees of freedom, and $\{\mdof_i\}_{i=1}^9$ are the moment type
degrees of freedom. The basis functions for $[P_2(T)]^2$ and $bP^\ast_2(T)$ are denoted by $\{\phi_i\}_{i=1}^{12}$ and $\{\psi_i\}_{i=1}^9$, respectively.

Define a new set of basis functions
\[
\varphi_i=\phi_i-\sum_{j=1}^9\mdof_j(\phi_i)\psi_j,\quad i=1,\cdots,12.
\]
We claim
\begin{equation}\label{eq:claim}
W(T)=\mr{span}\{\varphi_1,\cdots,\varphi_{12},\psi_1,\cdots,\psi_9\}.
\end{equation}
Note that
$\ldof_i(\psi_j)\equiv0$ because $\psi_j=0$ on $\pa T$. We obtain $\{\psi_j\}_{j=1}^9$ are the basis functions of $W(T)$ associated with the degrees of freedom $\{\mdof_j\}_{j=1}^9$. For any $\varphi_i$, there holds
\[
\ldof_j(\varphi_i)=\ldof_j(\phi_i)
-\sum_{k=1}^9\mdof_k(\phi_i)\ldof_j(\psi_k)=\ldof_j(\phi_i)=\delta_{ij},
\]
and
\begin{align*}
\mdof_j(\varphi_i)&=\mdof_j(\phi_i)-\sum_{k=1}^9\mdof_k(\phi_i)\mdof_j(\psi_k)\\
&=\mdof_j(\phi_i)-\sum_{k=1}^9\mdof_k(\phi_i)\delta_{jk}=0.
\end{align*}
This verifies the claim~\eqref{eq:claim}.

Next, we prove the interpolation operator is locally $P_2$ invariant. For any $v\in [P_2(T)]^2$, we have the representation
\[
v=\sum_{i=1}^{12}\ldof_i(v)\phi_i.
\]
By definition,
\begin{align*}
\Pi_T v&=\sum_{i=1}^{12}\ldof_i(v)\varphi_i+\sum_{j=1}^9\mdof_j(v)\psi_j\\
&=\sum_{i=1}^{12}\ldof_i(v)\phi_i-\sum_{j=1}^9
\sum_{i=1}^{12}\ldof_i(v)\mdof_j(\phi_i)\psi_j
+\sum_{j=1}^9\mdof_j(v)\psi_j\\
&=\sum_{i=1}^{12}\ldof_i(v)\phi_i-\sum_{j=1}^9\Lr{\sum_{i=1}^{12}
\ldof_i(v)\mdof_j(\phi_i)-\mdof_j(v)}\psi_j\\
&=\sum_{i=1}^{12}\ldof_i(v)\phi_i=v,
\end{align*}
where we have used the identity
\[
\mdof_j(v)=\mdof_j\Lr{\sum_{i=1}^{12}\ldof_i(v)\phi_i}
=\sum_{i=1}^{12}\ldof_i(v)\mdof_j(\phi_i).
\]
This completes the proof.
\qed
\end{proof}

The above proof actually provides a constructive way to derive the basis function of this element.
\subsection{The second nonconforming element}
The second element is almost the same with the first one except that $P_2^\ast$ is replaced by
\[
P_3^{\ast}(T)=\set{w\in[P_3(T)]^2}{\divop w\in P_0(T),w\cdot n|_{e}\in P_1(e)\;\text{for all\;} e\in\pa T}.
\]
Hence,
\begin{equation}\label{element:2}
W(T)=[P_2(T)]^2\oplus bP_3^{\ast}(T).
\end{equation}
Here $P_3^{\ast}(T)$ has appeared in~\cite{MardalTaiWinther:2002} to solve Darcy-Stokes flow.

The following lemma gives the degrees of freedom of this element, and we prove it is $W(T)-$unisolvent.
\begin{lemma}\label{lem:2ndunisolvency}
The dimension of $W(T)$ is $21$. Any $w\in W(T)$ is uniquely determined by the following degrees of freedom:
\begin{enumerate}
\item The values of $w$ at the corners and edge midpoints;
\item The moments $\int_e\pa_n (w\cdot t)\dtau$ and $\int_e\pa_n (w\cdot n) \tau^k\dtau$ for $k=0,1$
and for all $e\in\pa T$.
\end{enumerate}
\end{lemma}

The proof of this result is slightly different from the direct proof of Lemma~\ref{lem:unisolvent}.
\begin{proof}
Proceeding along the same line that leads to~\eqref{eq:1stnormalconstraint} and~\eqref{eq:1sttangential}, we obtain
that for all $e\in\pa T$,
\begin{equation}\label{eq:2ndconstraints}
\int_ep\cdot n\tau^k\dtau=0,k=0,1,\quad\text{and}\quad
\int_ep\cdot t\dtau=0.
\end{equation}
Using the fact that $(p\cdot n)|_{\pa T}\in P_1$ and the first identity of~\eqref{eq:2ndconstraints}, we conclude that $(p\cdot n)|_{\pa T}\equiv 0$, which immediately implies
\[
\int_T\divop p\dx=\int_{\pa T}p\cdot n\dtau=0.
\]
Since $\divop p\in P_0(T)$, this implies that $p$ is divergence free. Then there exists a polynomial $\phi\in P_4(T)$ such that $p=\curl\phi$. Furthermore, since
\[
\pa_t\phi|_{\pa T}=p\cdot n|_{\pa T}=0,
\]
which implies that $\phi$ is constant along the edge of $T$. Without loss of generality, we assume that $\phi|_{\pa T}\equiv0$. Hence, $\phi$ is of the form $\phi=b\,\varphi$ with $\varphi\in P_1(T)$. By the second
identity of~\eqref{eq:2ndconstraints}, we obtain
\[
\int_{e}\pa_n\phi\,\dtau=\int_ep\cdot t\,\dtau=0.
\]
Note that $\pa_n\phi|_e=\lam_+\lam_{-}\pa_n\lam_e\,\varphi|_e$ and the fact that
$\lam_+\lam_{-}\pa_n\lam_e$ is strictly negative in the interior of the edge $e$, we conclude that $\varphi$ has a root
at each edge, which together with the fact that $\varphi\in P_1(T)$ yields $\varphi\equiv0$, or equivalently $w=0$. This
completes the proof.\qed
\end{proof}

Proceeding along the same line that leads to~\eqref{eq:p2invariance}, we conclude that this
nonconforming element also locally preserves quadratics.

\begin{remark}
Both elements are endowed with the same degrees of freedom. The structure of the local finite element spaces
for both element are similar. In fact, the bubble functions can be removed by standard static condensation procedure.
Therefore, the resulting method has only $12$ degrees of freedom on each element.
\end{remark}

For any function $v\in V_h$, we obtain a simplified version of the discrete Korn's inequality~\eqref{discrete_Korn} without all the jump terms, which may be regarded as a H$^2-$ analog of the discrete Korn's inequality~\eqref{eq:h1korn}.
\begin{lemma}\label{coro:korn}
There exists $C$ depends on $\Om$ and $\mc{T}_h$, but independent of $h$ such that
\begin{equation}\label{eq:vkorn}
\nm{v}{H_h^2}\le C\Lr{\nm{\na_h\eps(v)}{L^2}+\nm{\eps(v)}{L^2}}.
\end{equation}
\end{lemma}

\begin{proof}
For any function $v\in V_h$, we claim that the jump terms in the right-hand side
of~\eqref{discrete_Korn} vanish. Indeed, $\jump{\Pi_e v}=0$ for any $e\in\mc{S}(\Om,\mc{T}_h)$
because $v\in [H_0^1(\Om)]^2$. It remains to verify that for all $e\in\mc{S}(\Om,\mc{T}_h)$ and $i=1,2$,
\begin{equation}\label{eq:zeronormaljump}
\jump{\Pi_e(\pa_iv)}=0.
\end{equation}
We write $\pa_iv=\al_i\pa_nv+\beta_i\pa_tv$, where $\al_i$ and $\beta_i$ are constants. Hence, it remains to show
\[
\jump{\Pi_e(\pa_nv)}=0,\qquad \jump{\Pi_e(\pa_tv)}=0,\quad\forall e\in\mc{S}(\Om,\mc{T}_h).
\]
Since $V_h$ is H$^1-$ conforming, it is clear that $\jump{\Pi_e(\pa_tv)}=0$.

For any $e\in\mc{S}(\Om,\mc{T}_h)$, it is clear that $\mathrm{RM}(e)=P_0(e)$. For any $w\in [P_{1,-}(e)]^2$, there holds
\[
w=w_n n+w_t t,\quad w_n\in P_1(e),w_t\in P_0(e).
\]
Hence, $\jump{\Pi_e(\pa_nv)}=0$, if and only if
\[
\int_e\jump{\pa_n(v\cdot n)}\tau^k\dtau=0,k=0,1,\quad\text{and}\quad
\int_e\jump{\pa_n(v\cdot t)}\dtau=0.
\]
This is true for any $v\in W(T)$ and we prove the claim~\eqref{eq:zeronormaljump}.

For any $v\in V_h$, it follows from~\eqref{discrete_Korn} that
\[
\nm{v}{H_h^2}\le C\Lr{\nm{\na_h\eps(v)}{L^2}+\nm{\eps(v)}{L^2}+\nm{v}{L^2}}.
\]
The inequality~\eqref{eq:vkorn} follows by using the Poincar\'e's inequality and the first
Korn's inequality~\eqref{eq:1stkorn}:
\[
\nm{v}{L^2}^2\le C_p^2\nm{\na v}{L^2}^2\le 2C_p^2\nm{\eps(v)}{L^2}^2.
\]
\qed
\end{proof}

We are ready to prove the coercivity of the bilinear form $a_h$ over $V_h$.
\begin{theorem}\label{th:discoer}
For any $\iota<1/\sqrt2$, there exists $C$ that depends on the domain $\Om$ and the shape regularity of the triangulation $\mc{T}_h$ such that
\begin{equation}\label{eq:coercivity}
a_h(v,v)\ge C(\iota^2\nm{v}{H_h^2}^2+\nm{v}{H^1}^2),\qquad\forall v\in V_h.
\end{equation}
\end{theorem}

\begin{proof}
For any $v\in V_h$, using~\eqref{eq:vkorn}, we obtain, there exists $C$ that is independent of $h$ such that
\begin{align*}
a_h(v,v)&\ge 2\mu\Lr{\iota^2\nm{\na_h\eps(v)}{L^2}^2+\nm{\eps(v)}{L^2}^2}\\
&\ge 2\mu\iota^2\Lr{\nm{\na_h\eps(v)}{L^2}^2+\nm{\eps(v)}{L^2}^2}
+\mu\nm{\eps(v)}{L^2}^2\\
&\ge C\Lr{\iota^2\nm{v}{H_h^2}^2+\nm{v}{H^1}^2},
\end{align*}
which implies~\eqref{eq:coercivity}, where we have used
\[
\nm{v}{H^1}^2=\nm{v}{L^2}^2+\nm{\na v}{L^2}^2\le(C_p^2+1)\nm{\na v}{L^2}^2\le 2(C_p^2+1)\nm{\eps(v)}{L^2}^2,
\]
in the last step.\qed
\end{proof}

The following interpolate estimate is a direct consequence of the quadratics invariance of
the local finite element spaces $W(T)$; The proof is standard, and we refer to~\cite{CiarletRaviart:1972} for the details.
\begin{lemma}\label{lem:interest}
There exists $C$ independent of $h$ such that for all $v\in[H^k(T)]^2$,
\begin{equation}\label{eq:interest}
\nm{v-\Pi_Tv}{H^j(T)}\le C h^{k-j}\abs{v}_{H^k(T)},
\quad j=0,1,2,k=2,3.
\end{equation}
\end{lemma}

A global interpolation operator $I_h:H^k(\Om)\mapsto V_h$ is defined by $(I_h)|_T=\Pi_T$.
\subsection{Convergence analysis}
We are ready to prove the main result of this paper.
\begin{theorem}\label{thm:error}
Assume that the weak solution of $u$ of the problem~\eqref{variation:eq} belongs to $[H_0^2(\Om)]^2\cap[H^3(\Om)]^2$. Let $u_h$ be the solution of~\eqref{eq:discvara}. Then there exists $C$
independent of $\iota$ and $h$ such that
\begin{equation}\label{eq:finalerr}
\wnm{u-u_h}\le\left\{
\begin{aligned}
&C(h^2+\iota h)\abs{u}_{H^3},\\
&Ch(\abs{u}_{H^2}+\iota\abs{u}_{H^3}),
\end{aligned}\right.
\end{equation}
where $\wnm{v}^2{:}=a_h(v,v)$ for any $v\in V_h$.
\end{theorem}

\begin{proof}
By the the theorem of Berger, Scott, and Strang~\cite{Berge:1972}, we have
\begin{equation}\label{estimate:2}
\wnm{u-u_h}\le\inf_{v\in V_h}\wnm{u-v}
+\sup_{w\in V_h}\dfrac{E_h(u,w)}{\wnm{w}},
\end{equation}
where $E_h(u,w)=a_h(u,w)-(f,w)$.

By the interpolate estimate~\eqref{eq:interest}, we obtain
\begin{equation}\label{estimate:3}
\inf_{v\in V_h}\wnm{u-v}\le \wnm{u-I_h u}\le\left\{
\begin{aligned}
& C(h^2+\iota h)\abs{u}_{H^3},\\
& Ch(\abs{u}_{H^2}+\iota\abs{u}_{H^3}).
\end{aligned}\right.
\end{equation}

Next, we focus on the estimate of the consistency error. We write
$\ka_{ijk}=(\na\eps(u))_{ijk}=\pa_{x_i}\eps_{jk}(u)$. The stress and couple stress are defined by $\sigma=\C\eps(u)$ and $\tau=\D\na\eps(u)$, respectively. Or
\[
\sigma_{ij}=\C_{ijkl}\eps_{kl}(u)\quad\text{and}\quad
\tau_{ijk}=\D_{ijklmn}\ka_{lmn}(u).
\]
By the symmetry of the tensors $\C$ and $\D$, there holds
\[
\sigma_{ij}=\sigma_{ji}\quad\text{and}\quad\tau_{ijk}=\tau_{ikj}.
\]
By the chain rule and the symmetry of $\C$ and $\D$, we obtain, on each element $T$
and for any $w\in V_h$,
\begin{align*}
&\quad\C\eps(u):\eps(w)+\D\na\eps(u):\na\eps(w)\\
&=\sigma_{jk}\eps_{jk}(w)+\tau_{ijk}\ka_{ijk}(w)=\sigma_{jk}w_{k,j}+\tau_{ijk}w_{k,ij}\\
&=\Lr{(\sigma_{jk}-\tau_{ijk,i})w_k}_{,j}-(\sigma_{jk,j}-\tau_{ijk,ij})w_k
+(\tau_{ijk}w_{k,j})_{,i}.
\end{align*}
Using the above representation and integration by parts, we obtain
\begin{equation}\label{int_1}
\begin{aligned}
\int_T\mathbb{C}\eps(u):\eps(w)&+\mathbb{D}\na\eps(u):\na\eps(w)\dx
=\int_T(\tau_{ijk,ij}-\sigma_{jk,j})w_k\dx\\
&+\int_{\pa T}n_j(\sigma_{jk}-\tau_{ijk,i})w_k\dtau+\int_{\pa T}n_i\tau_{ijk}w_{k,j}\dtau.
\end{aligned}
\end{equation}
Using the fact \(w_{k,j}=n_j\pa_nw_k+t_j\pa_tw_k\) and
\[
n_it_j\tau_{ijk}\pa_tw_k=\pa_t(n_it_j\tau_{ijk}w_k)-\pa_t(n_it_j\tau_{ijk})w_k,
\]
we obtain
\begin{equation}\label{int_2}
\int_{\pa T}n_i\tau_{ijk}w_{k,j}\dtau=\int_{\pa T}n_in_j\tau_{ijk}\pa_nw_k\dtau-\int_{\pa T}(n_it_j\tau_{ijk})w_k\dtau,
\end{equation}
where we have used the fact that the contour integration of tangential derivative along the element boundary is zero.
By~\eqref{int_1},~\eqref{int_2} and the continuity of $w$, we obtain
\[
E_h(u,w)=\sum_{T\in\mc{T}_h}\int_{\pa T}n_in_j\tau_{ijk}\pa_nw_k\dtau
=\sum_{e\in\mc{S}(\Om,\mc{T}_h)}\int_{e} n_in_j\tau_{ijk}\jump{\pa_n w_k}\dtau,
\]
where $\tau_{ijk}=\iota^2\sigma_{jk,i}$.

By
\[
\int_e\jump{\pa_n(w\cdot n)}\dtau=0\quad\text{and}\quad
\int_e\jump{\pa_n(w\cdot t)}\dtau=0,
\]
we obtain, for $k=1,2$,
\[
\int_e\jump{\pa_n w_k}\dtau=0.
\]
Employing the standard trace inequality and scaling argument, we obtain
\[
\abs{E_h(u,w)}\le Ch\abs{\tau}_{H^1}\abs{\pa_n w}_{H_h^1}\le C\iota^2 h\abs{u}_{H^3}\abs{w}_{H_h^2}.
\]
Substituting the above estimate and~\eqref{estimate:3} into~\eqref{estimate:2}, we obtain~\eqref{eq:finalerr}.\qed
\end{proof}

Combining the error estimate~\eqref{eq:finalerr} and the regularity results in Lemma~\ref{lemma:reg}, proceeding along the same line of~\cite[Theorem 5.2]{Tai:2001}, we could obtain the following $\iota-$independent error estimate
\begin{equation}\label{eq:newerr}
\wnm{u-u_h}\le Ch^{1/2}\nm{f}{L^2},
\end{equation}
where $C$ is independent of $\iota$ and $h$. We leave the details to the interested readers.
\section{Numerical example}~\label{sec:numer}
In this section we provide two numerical examples that show the accuracy of the proposed elements, and the robustness
of the elements with respect to the microscaopic parameter $\iota$. The first example is performed on the
uniform mesh, while the second one is performed on the nonuniform mesh. As a first
step toward understanding the size effect of the heterogeneous materials, we test
the proposed elements for a benchmark problem with smooth solution. Tests for realistic problems
will appear in the forthcoming work.

Let $\Om=(0,1)^2$ and
\[
u_1=(\exp(\cos2\pi x)-e)(\exp(\cos2\pi y)-e), u_2=(\cos2\pi x-1)(\cos4\pi y -1).
\]
The force $f$ is obtained by~\eqref{eq:sgbvp}.

First, the triangulation of the unit square for the uniform mesh is illustrated in Figure~\ref{pic:1}$_a$. In Table~\ref{tab:1} and Table~\ref{tab:2}, we report the convergence rates for both
elements in the energy norm $\wnm{u-u_h}/\wnm{u}$ for different values of $\iota$ and $h$ with $\lam=\mu=1$. We observe that the convergence rate appears to be linear when $\iota$ is large, while it turns out to be quadratic when $\iota$ is close to zero, which is consistent with the theoretical prediction in the estimate~\eqref{eq:finalerr}.
\begin{table}[htbp]
\caption{The convergence rate of the first element over uniform mesh with $\lambda=\mu=1$.}\label{tab:1}
\begin{tabular}{lllllll}
\hline\noalign{\smallskip}
$\iota\backslash h$ &$1/16$ &$1/32$ &$1/64$ &$1/128$ &$1/256$ &$1/512$\\
\hline\noalign{\smallskip}
1e0 & 2.37e-1 &1.38e-1 &7.31e-2 &3.73e-2 &1.87e-2 &9.38e-3\\
\noalign{\smallskip}
\text{rate}& &0.79&0.91&0.97&0.99&1.00\\
\noalign{\smallskip}
1e-1 &1.81e-1 &1.04e-1 &5.47e-2 &2.78e-2 &1.40e-2 &6.99e-3\\
\noalign{\smallskip}
\text{rate}&&0.81&0.92&0.98&0.99&1.00\\
\noalign{\smallskip}
1e-2 &3.44e-2 &1.66e-2 &8.28e-3 &4.15e-3 &2.07e-3 &1.04e-3\\
\noalign{\smallskip}
\text{rate}&&1.06&1.00&1.00&1.00&1.00\\
\noalign{\smallskip}
1e-3 &1.87e-2 &5.25e-3 &1.54e-3 &5.35e-4 &2.26e-4 &1.07e-4\\
\noalign{\smallskip}
\text{rate}& &1.83&1.77&1.53&1.25&1.08\\
\noalign{\smallskip}
1e-4 &1.85e-2 &4.96e-3 &1.27e-3 &3.22e-4 &8.30e-5 &2.29e-5\\
\noalign{\smallskip}
\text{rate}&&1.90&1.96&1.98&1.96&1.86\\
\noalign{\smallskip}
1e-5 &1.85e-2 &4.95e-3 &1.27e-3 &3.19e-4 &7.98e-5 &2.00e-5\\
\noalign{\smallskip}
\text{rate}&&1.90&1.97&1.99&2.00&2.00\\
\noalign{\smallskip}
\hline
\end{tabular}
\end{table}
\begin{table}[htbp]
\caption{The convergence rate of the second element over uniform mesh with $\lambda=\mu=1$.}\label{tab:2}
\begin{tabular}{lllllll}
\hline\noalign{\smallskip}
$\iota\backslash h$ &$1/16$ &$1/32$ &$1/64$ &$1/128$ &$1/256$ &$1/512$ \\
\hline\noalign{\smallskip}
1e0 &2.73e-1 &1.67e-1 &9.22e-2 &4.76e-2 &2.40e-2 &1.20e-2\\
\noalign{\smallskip}
\text{Rate}&&0.70&0.86&0.95&0.99&1.00\\
\noalign{\smallskip}
1e-1 &2.10e-1 &1.27e-1 &6.91e-2 &3.55e-2 &1.79e-2 &8.97e-3\\
\noalign{\smallskip}
\text{Rate}&&0.73&0.88&0.96&0.99&1.00\\
\noalign{\smallskip}
1e-2 &4.01e-2 &2.04e-2 &1.05e-2 &5.30e-3 &2.66e-3 &1.33e-3\\
\noalign{\smallskip}
\text{Rate}&&0.97&0.96&0.98&1.00&1.00\\
\noalign{\smallskip}
1e-3 &2.13e-2 &6.10e-3 &1.83e-3 &6.54e-4 &2.84e-4 &1.36e-4\\
\noalign{\smallskip}
\text{Rate}&&1.80&1.74&1.48&1.20&1.06\\
\noalign{\smallskip}
1e-4 &2.10e-2 &5.73e-3 &1.48e-3 &3.75e-4 &9.70e-5 &2.70e-5\\
\noalign{\smallskip}
\text{Rate}&&1.87&1.95&1.98&1.95&1.85\\
\noalign{\smallskip}
1e-5 &2.10e-2 &5.73e-3 &1.47e-3 &3.71e-4 &9.29e-5 &2.33e-5\\
\noalign{\smallskip}
\text{Rate}&&1.87&1.96&1.99&2.00&2.00\\
\hline
\end{tabular}
\end{table}

Next, we test both elements over a nonuniform mesh. The initial mesh is generated by the function "initmesh" of the partial differential equation toolbox of MATLAB. The initial mesh consists of $872$ triangles and $469$ vertices, and the
maximum mesh size is $h=1/16$; See Figure~\ref{pic:1}$_b$. In Tables~\ref{tab:3},~\ref{tab:4},~\ref{tab:5} and~\ref{tab:6}, we report the convergence rate of both elements in the energy norm when $\lambda=\mu=1$ and $\lambda=10,\mu=1$. It seems the convergence rate is the same with that over the uniform mesh. The first element is slightly more accurate than the second one, in particular over the nonuniform mesh.%
\begin{figure}[htbp]
\centering
\subfigure[]{\includegraphics[width=5.5cm, height=4.2cm]{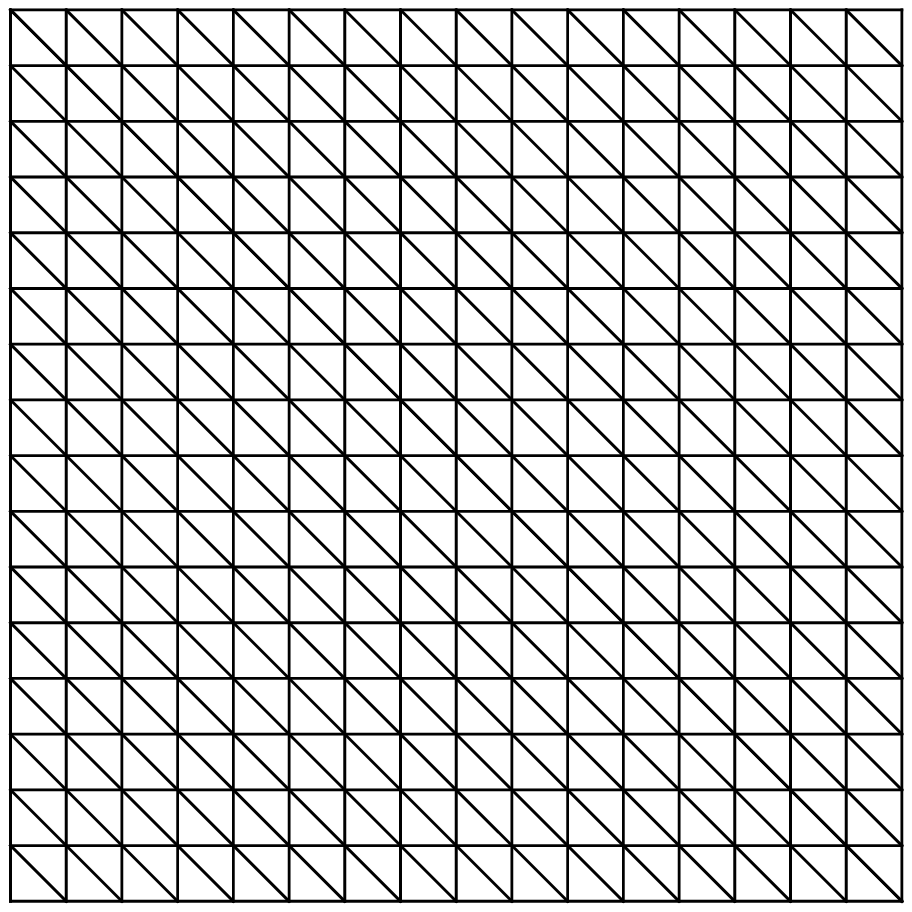}}
\subfigure[]{\includegraphics[width=5.5cm, height=4.2cm]{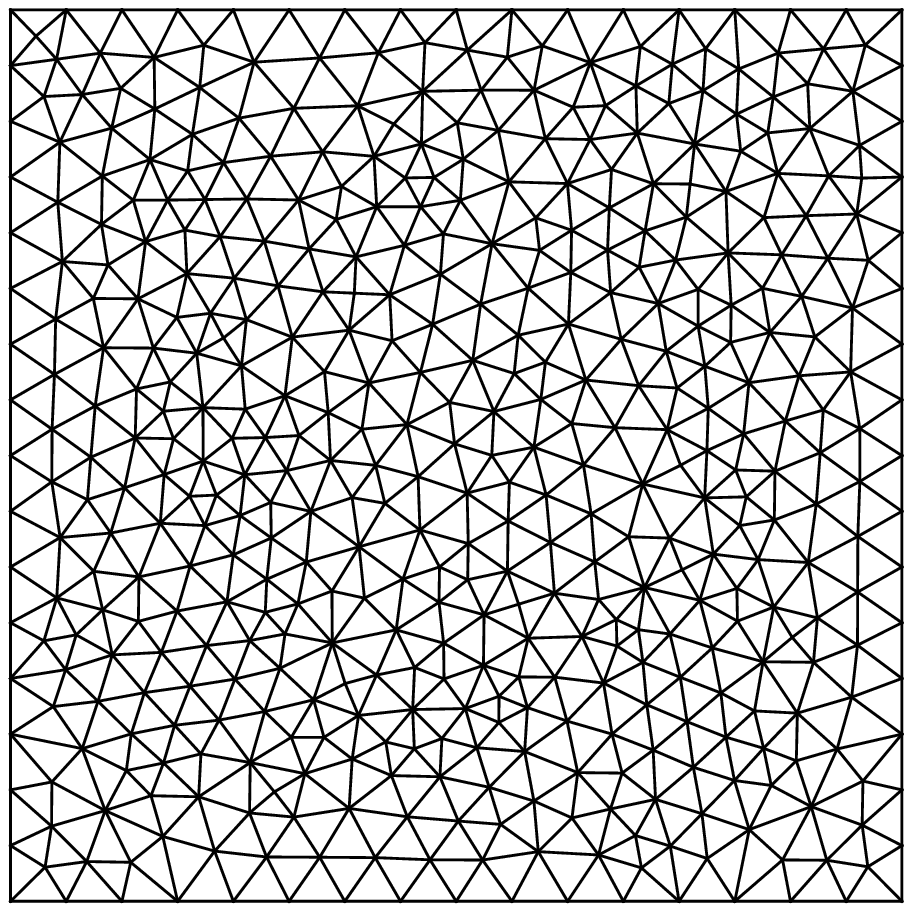}}
\caption{Plots of the mesh. (a) is the uniform triangulations with $h=1/16$. (b) is the nonuniform mesh with maximum mesh size $h=1/16$.}\label{pic:1}
\end{figure}
\begin{table}[htbp]
\caption{The convergence rate of the first element over nonuniform mesh with $\lambda=\mu=1$.}\label{tab:3}
\begin{tabular}{lllllll}
\hline
$\iota\backslash h$ &$1/16$ &$1/32$ &$1/64$ &$1/128$ &$1/256$&$1/512$ \\
\hline\noalign{\smallskip}
1e0 &1.26e-1&6.80e-2&3.61e-2&1.87e-2&9.57e-3&4.84e-3\\
\noalign{\smallskip}
\text{Rate}&&0.89&0.91&0.94&0.97&0.98\\
\noalign{\smallskip}
1e-1&9.43e-2&5.08e-2&2.69e-2&1.40e-2&7.14e-3&3.61e-03\\
\noalign{\smallskip}
\text{Rate}&&0.89&0.92&0.95&0.97&0.98\\
\noalign{\smallskip}
1e-2&1.66e-2&7.92e-3&4.05e-3&2.08e-3&1.06e-3&5.35e-4\\
\noalign{\smallskip}
\text{Rate}&&1.07&0.97&0.96&0.97&0.99\\
\noalign{\smallskip}
1e-3&8.24e-3&2.24e-3&6.79e-4&2.52e-4&1.12e-4&5.46e-5\\
\noalign{\smallskip}
\text{Rate}&&1.88&1.73&1.43&1.16&1.04\\
\noalign{\smallskip}
1e-4&8.09e-3&2.09e-3&5.33e-4&1.36e-4&3.56e-5&1.01e-5\\
\noalign{\smallskip}
\text{Rate}&&1.95&1.97&1.97&1.94&1.81\\
\noalign{\smallskip}
1e-5&8.09e-3&2.08e-3&5.31e-4&1.34e-4&3.38e-5&8.48e-6\\
\noalign{\smallskip}
\text{Rate}&&1.96&1.97&1.98&1.99&1.99\\
\hline
\end{tabular}
\end{table}
\begin{table}[htbp]
\caption{The convergence rate of the first element over nonuniform mesh with $\lambda=10$ and $\mu=1$.}\label{tab:4}
\begin{tabular}{lllllll}
\hline
$\iota\backslash h$ &$1/16$ &$1/32$ &$1/64$ &$1/128$ &$1/256$&$1/512$ \\
\hline\noalign{\smallskip}
1e0&1.09e-1&5.87e-2&3.12e-2&1.63e-2&8.36e-3&4.24e-3\\
\noalign{\smallskip}
\text{Rate}&&0.89&0.91&0.94&0.96&0.98\\
\noalign{\smallskip}
1e-1&8.22e-2&4.43e-2&2.35e-2&1.23e-2&6.30e-3&3.19e-3\\
\noalign{\smallskip}
\text{Rate}&&0.89&0.91&0.94&0.96&0.98\\
\noalign{\smallskip}
1e-2&1.43e-2&6.93e-3&3.58e-3&1.85e-3&9.47e-4&4.79e-4\\
\noalign{\smallskip}
\text{Rate}&&1.05&0.96&0.95&0.97&0.98\\
\noalign{\smallskip}
1e-3&6.67e-3&1.82e-3&5.66e-4&2.18e-4&9.96e-5&4.88e-5\\
\noalign{\smallskip}
\text{Rate}&&1.87&1.69&1.38&1.13&1.03\\
\noalign{\smallskip}
1e-4&6.52e-3&1.67e-3&4.25e-4&1.09e-4&2.87e-5&8.40e-6\\
\noalign{\smallskip}
\text{Rate}&&1.97&1.97&1.97&1.92&1.77\\
\noalign{\smallskip}
1e-5&6.52e-3&1.67e-3&4.23e-4&1.07e-4&2.69e-5&6.75e-6\\
\noalign{\smallskip}
\text{Rate}&&1.97&1.98&1.99&1.99&1.99\\
\hline
\end{tabular}
\end{table}
\begin{table}[htbp]
\caption{The convergence rate of the second element over nonuniform mesh with $\lambda=\mu=1$.}\label{tab:5}
\begin{tabular}{lllllll}
\hline
$\iota\backslash h$ &$1/16$ &$1/32$ &$1/64$ &$1/128$ &$1/256$&$1/512$ \\
\hline\noalign{\smallskip}
1e0&2.50e-1&1.44e-1&8.06e-2&4.29e-2&2.20e-2&1.11e-02\\
\noalign{\smallskip}
\text{Rate}&&0.80&0.84&0.91&0.96&0.98\\
\noalign{\smallskip}
1e-1&1.88e-1&1.08e-1&6.02e-2&3.20e-2&1.64e-2&8.30e-3\\
\noalign{\smallskip}
\text{Rate}&&0.80&0.84&0.91&0.96&0.98\\
\noalign{\smallskip}
1e-2&3.14e-2&1.65e-2&9.04e-3&4.77e-3&2.44e-3&1.23e-3\\
\noalign{\smallskip}
\text{Rate}&&0.92&0.87&0.92&0.97&0.99\\
\noalign{\smallskip}
1e-3&1.34e-2&3.79e-3&1.24e-3&5.11e-4&2.48e-4&1.24e-4\\
\noalign{\smallskip}
\text{Rate}&&1.82&1.62&1.27&1.04&1.00\\
\noalign{\smallskip}
1e-4&1.31e-2&3.43e-3&8.85e-4&2.28e-4&6.06e-5&1.82e-5\\
\noalign{\smallskip}
\text{Rate}&&1.93&1.96&1.96&1.91&1.74\\
\noalign{\smallskip}
1e-5&1.31e-2&3.43e-3&8.80e-4&2.23e-4&5.61e-5&1.41e-5\\
\noalign{\smallskip}
\text{Rate}&&1.93&1.96&1.98&1.99&1.99\\
\hline
\end{tabular}
\end{table}
\begin{table}[htbp]
\caption{The convergence rate of the second element over nonuniform mesh with $\lambda=10$ and $\mu=1$.}\label{tab:6}
\begin{tabular}{lllllll}
\hline
$\iota\backslash h$ &$1/16$ &$1/32$ &$1/64$ &$1/128$ &$1/256$&$1/512$ \\
\hline
1e0&2.47e-1&1.35e-1&7.41e-2&3.99e-2&2.09e-2&1.07e-2\\
\noalign{\smallskip}
\text{Rate}&&0.87&0.86&0.89&0.93&0.97\\
\noalign{\smallskip}
1e-1&1.87e-1&1.02e-1&5.59e-2&3.00e-2&1.58e-2&8.07e-3\\
\noalign{\smallskip}
\text{Rate}&&0.88&0.86&0.89&0.93&0.97\\
\noalign{\smallskip}
1e-2&3.19e-2&1.59e-2&8.53e-3&4.55e-3&2.37e-3&1.21e-3\\
\noalign{\smallskip}
\text{Rate}&&1.01&0.90&0.90&0.94&0.97\\
\noalign{\smallskip}
1e-3&1.43e-2&3.92e-3&1.25e-3&5.04e-4&2.43e-4&1.23e-4\\
\noalign{\smallskip}
\text{Rate}&&1.87&1.65&1.31&1.05&0.99\\
\noalign{\smallskip}
1e-4&1.41e-2&3.63e-3&9.32e-4&2.39e-4&6.34e-5&1.88e-5\\
\noalign{\smallskip}
\text{Rate}&&1.96&1.96&1.96&1.92&1.75\\
\noalign{\smallskip}
1e-5&1.41e-2&3.63e-3&9.29e-4&2.35e-4&5.92e-5&1.49e-5\\
\noalign{\smallskip}
\text{Rate}&&1.96&1.97&1.98&1.99&1.99\\
\hline
\end{tabular}
\end{table}
\section{Conclusion}
We prove a Korn-like inequality and its discrete analog for the strain gradient elastic problem, which is crucial for the well-posedness of the underlying variational problems as the Korn's
inequality for the linearized elasticity. Guided by the discrete Korn's inequality, we constructed two nonconforming
elements that converge uniformly in the microscopic parameter with optimal convergence rate. Numerical
experiments validate the theoretical results. The extension of the elements to three dimensional
problem and to high order would be very interesting and challenging. Applications of these elements to realistic problem in strain gradient plasticity is another topic deserves further pursuit. We leave all these issues in a forthcoming work.
\begin{acknowledgements}
The work of Li was supported by Science Challenge Project, No. TZ 2016003. The work of Ming was partially supported by the National Natural Science Foundation of China for
Distinguished Young Scholars 11425106, and National Natural Science Foundation of China grants 91630313, and by the support of CAS NCMIS. The work of Shi
was partially supported by the National Natural Science Foundation of China grant 11371359. We are grateful to the anonymous referees for their valuable suggestions.
\end{acknowledgements}
\bibliographystyle{spbasic}
\bibliography{sg}
\end{document}